\documentclass[12pt,a4paper,twoside]{article}

\usepackage{amsmath, amssymb}
\usepackage{ascmac}

\DeclareMathOperator{\Pic}{Pic}

\DeclareMathOperator{\rk}{rk}
\begin{document}

\title{\textbf{\Large{Second gonality of smooth aCM curves on quartic surfaces in $\mathbb{P}^3$}}}

\author{Kenta Watanabe \thanks{Nihon University, College of Science and Technology,   7-24-1 Narashinodai Funabashi city Chiba 274-8501 Japan , {\textit E-mail address:watanabe.kenta@nihon-u.ac.jp}}}

\date{}

\maketitle 

\begin{abstract}

\noindent For a smooth irreducible curve $C$, its second gonality $d_2$ is defined to be the minimum 
integer $d$ such that $C$ admits a linear series $g_d^2$. In this paper, we compute the second gonality of a smooth aCM curve $C$ lying on a smooth quartic surface in $\mathbb{P}^3$, whose Clifford index is computed by a net on $C$.
 
\end{abstract}

\noindent {\textbf{Keywords}} K3 surface, Brill-Noether theory, LM bundle, Clifford index, second gonality

\smallskip

\noindent {\textbf{Mathematics Subject Classification }}14J28, 14J60, 14H60

\section{Introduction}

Let $C$ be a smooth projective curve of genus $g \geq 4$, and let $d_r$ be the minimum degree of a linear series of dimension $r$ on $C$. We call $d_r$ the $r${\textit{-th gonality}} of $C$, and call the sequence $\{d_r\}_{r \in \mathbb{N}}$ the {\textit{gonality sequence}} of $C$. In particular, $d_1$ corresponds to the gonality of $C$ in the classical sense. For $r \geq g$, $d_r$ is determined by the Riemann-Roch theorem. Furthermore, the gonality sequences of several types of curves with small gonality have been computed by H. Lange, G. Martens, and P. E. Newstead ([3], [4]).

For a base-point-free line bundle $L$ of degree $d_r$ on $C$, given a subspace $V \subset H^0(L)$ that forms an $r$-dimensional base-point-free sub-linear system of $|L|$, let $M_{V,L}$ be the vector bundle on $C$ which is defined as the dual of the kernel of the evaluation map $ev : V \otimes \mathcal{O}_C \longrightarrow L$. If $C$ is a Petri general curve, stable bundles on $C$ are obtained as extensions of vector bundles of type $M_{V,L}$. Consequently, in the context of the Mercat conjecture in higher-rank Brill-Noether theory, the problem concerning the stability of vector bundles on a Petri general curve often reduces to investigating the stability of such bundles ([3]). Since the slope of the vector bundle $M_{V,L}$ is given by $d_r/r$, whether $M_{V,L}$ is stable reduces to examining the inequalities among the slopes associated with the filtration consisting of sub-bundles of $M_{V,L}$. If $d_r$ is the $r$-th gonality of $C$, it is generally expected that the inequality
$$\frac{d_{r-1}}{r-1} \geq \frac{d_{r}}{r}$$
holds for each $r \geq 2$. This inequality is referred to as the slope inequality. The slope inequality is satisfied if $d_r$ does not increase too rapidly with respect to $r$. Accordingly, Lange and Martens carefully examined such cases and demonstrated that the slope inequality fails for general curves in $\mathbb{P}^r$, curves on general K3 surfaces in $\mathbb{P}^r$, and complete intersection curves in $\mathbb{P}^3$. Specifically, Lange and Martens showed that if there exists a linear system $g_d^r$ providing an embedding of degree $d$ into $\mathbb{P}^r$ such that $d_{r-1} = d-1$, then its Serre dual linear system $|K_C - g_d^r|$ causes the collapse of the slope inequality ([4, Lemma 4.8]). All counterexamples provided by Lange and Martens are projectively normal curves. Therefore, it is of general interest to classify projectively normal curves of degree $d$ satisfying $d_{r-1} = d-1$. The purpose of this paper is to compute the second gonality of smooth aCM curves in $\mathbb{P}^3$ by utilizing the classification of aCM line bundles on smooth quartic surfaces in $\mathbb{P}^3$ [10, Theorem 1.1]. Our main result is the following.

\newtheorem{thm}{Theorem}[section]

\begin{thm} Let $X$ be a smooth quartic surface in $\mathbb{P}^3$, and let $H$ be the line bundle on $X$ associated with a hyperplane section of $X$. Let $C\subset X$ be a smooth aCM curve of genus $g\geq9$ and degree $d$ with respect to $H|_C$. If there exists a linear series $g_{d_2}^2$ on $C$ with $g>2d_2-2$ which computes the Clifford index of $C$, then $d_2=d-1$. \end{thm}

\smallskip

\noindent By the definition of the second gonality of a curve $C$, any linear series of dimension two which computes the Clifford index of $C$ attains the second gonality of $C$. Hence, the second gonality of any smooth aCM curve of degree $d$ satisfying the hypothesis of Theorem 1.1 is $d-1$. 

Our plan of this paper is as follows. In section 2, we recall the definition of the gonality sequence of a curve, its properties, and the result concerning the classification of aCM line bundles on a smooth quartic surface in $\mathbb{P}^3$. In section 3, we recall some fundamental facts concerning vector bundles on K3 surfaces. In section 4, we recall the definition of the Lazarsfeld-Mukai bundle associated with a smooth curve on a K3 surface and a base-point-free linear series on it, and several known results about it. In section 5, we will give a proof of Theorem 1.1, and also give an example of a smooth aCM curve of genus 9 and degree $d$ in $\mathbb{P}^3$ whose second gonality $d_2$ satisfies $d_2<d-1$.

\smallskip

\smallskip

\noindent{\textbf{Notations and conventions}}. Throughout this paper, we work over the field of complex numbers $\mathbb{C}$. Unless otherwise stated, all curves and surfaces are assumed to be smooth and projective. For a variety $X$, we denote its canonical line bundle by $K_X$. For divisors $D_1$ and $D_2$ on $X$, linear equivalence is denoted by $D_1 \sim D_2$. The linear system associated with a divisor or a line bundle $L$ is denoted by $|L|$. A linear series of dimension one is referred to as a pencil, and a linear series of dimension two is referred to as a net. For a torsion-free sheaf $E$ on $X$, we denote its rank, its dual, and its $i$-th Chern class by $\rk(E)$, $E^{\vee}$, and $c_i(E)$, respectively.

For a surface $X$, we denote its Picard group by $\text{Pic}(X)$ and its rank by $\rho$. We call it  the Picard number of $X$. The Hodge index theorem states that the intersection form on $\text{Pic}(X)$ has signature $(1, \rho-1)$. As a consequence, the inequality $(D_1 \cdot D_2)^2 \geq D_1^2 D_2^2$ holds for any two divisors $D_1, D_2$ on $X$ with $D_1^2 > 0$ and $D_2^2 > 0$. For a line bundle $L$ and a divisor $D$ on $X$, we write $L(D) := L \otimes \mathcal{O}_X(D)$. Frequently, we identify line bundles with their associated divisors when no confusion arises. The group operation in $\text{Pic}(X)$ is written additively, and the dual of a line bundle $L$ is denoted by $-L$. The restriction of $L$ to a curve $C \subset X$ is denoted by $L|_C$. For a subscheme $W \subseteq X$, its ideal sheaf is denoted by $\mathcal{J}_W$. We say that $X$ is a regular surface if $h^1(\mathcal{O}_X) = 0$. In particular, a regular surface with a trivial canonical bundle $K_X$ is called a K3 surface.

\section{Preliminaries}

In this section, we recall the definitions of the gonality sequence and the Clifford index of a curve, and the properties of them. Moreover, we mention the relationship between the notions of aCM curves and extremal curves in a projective space.

Let $C$ be a smooth irreducible curve of genus $g$. For $r\in\mathbb{N}$, the {\textit{r-th gonality}} of $C$ is defined as follows:
$$d_r:=\min\{d\;|\;C\text{ admits a linear series }g_d^r\}.$$
In particular, $d_1$ is the gonality of $X$. The sequence $\{d_r\}_{r\in\mathbb{N}}$ is called the {\textit{gonality sequence}} of $C$. For a linear series $g_d^r$ on $C$, we denote its {\textit{Clifford index}} by $\gamma(g_d^r):=d-2r$. If there exists a line bundle $L$ on $C$ with $g_d^r=|L|$, (that is, $g_d^r$ is complete), then $\gamma(L)=\gamma(K_C\otimes L^{\vee})$. Hence, by the Riemann-Roch theorem and the Serre duality, without loss of generality, we may assume that $d\leq g-1$. By extending this criterion to linear series on $C$, the Clifford index of $C$ is defined as the minimum value of the Clifford indices of linear series $g_d^r$ with $r\geq1$ and $d\leq g-1$, and we often denote it by $\gamma_C$. Moreover, we say that a linear series $g_d^r$ {\textit{contributes to $\gamma_C$}} if $r\geq1$, and $d\leq g-1$. 

$\;$

\noindent{\textbf{Remark 2.1}}. We can easily see that if $g_d^r$ computes $\gamma_C$, then it is complete and base-point-free, and $d=d_r$.

$\;$

Assume that $C$ is embedded into $\mathbb{P}^r$. Then we call $C$ an {\textit{arithmetically Cohen-Macaulay}} ({\textit{aCM}} for short) curve if $h^1(\mathcal{I}_C(l))=0$, for all $l\in\mathbb{Z}$, where $\mathcal{I}_C$ is the ideal of $C$ in $\mathbb{P}^r$. If $C$ is contained in a smooth surface $X$ of degree $m$ in $\mathbb{P}^3$, then, by the exact sequence
$$0\longrightarrow\mathcal{O}_{\mathbb{P}^3}(-m)\longrightarrow\mathcal{I}_C\longrightarrow\mathcal{O}_X(-C)\longrightarrow0,$$
$\mathcal{O}_X(C)$ is aCM if and only if $C$ is aCM. In particular, we may obtain the characterization of all aCM curves on a smooth quartic surface in $\mathbb{P}^3$, by the following proposition.

\newtheorem{prop}{Proposition}[section]

\begin{prop} {\rm{([10, Theorem 1.1])}} Let $X$ be a smooth quartic surface in $\mathbb{P}^3$, let $H$ be a hyperplane section of $X$, and let $D$ be a divisor on $X$. Then the following conditions are equivalent.

\smallskip

\smallskip

\noindent {\rm{(i)}} $\mathcal{O}_X(D)$ is aCM and initialized.

\noindent {\rm{(ii)}} One of the following cases occurs.

\smallskip

\smallskip

{\rm{(a)}} $D^2=-2$ and $H.D\in\{1,2,3\}$.

{\rm{(b)}} $D^2=0$ and $H.D\in\{0,3,4\}$.

{\rm{(c)}} $D^2=2$ and $H.D=5$.

{\rm{(d)}} $D^2=4,\;H.D=6$, and $|D-H|=|2H-D|=\emptyset.$\end{prop}

\noindent{\textbf{Remark 2.2}}. If $C$ is an aCM curve on a smooth quartic surface $X$ in $\mathbb{P}^3$, then there exist an effective divisor $D$ satisfying the condition (ii) as in Proposition 2.1, and $s\in\mathbb{N}\cup\{0\}$ such that $C\sim sH+D$.

\smallskip

\smallskip

If $g\geq1$ and $g_d^r$ is a very ample linear system on $C$, then $d\geq r+1$, and, by the Castelnuovo's genus bound, we get the following inequality.
$$g\leq k\biggl(\frac{k-1}{2}(r-1)+\epsilon \biggr),\quad (2.1)$$
where $d-1=k(r-1)+\epsilon$ and $0\leq\epsilon<r-1$. If the upper bound of the inequality (2.1) is attained, then $C\subset\mathbb{P}^r$ is called an {\textit{extremal curve}}. Any extremal curve treated in [4] is aCM. However, a smooth aCM curve is not necessarily extremal. For instance, any smooth aCM curve $D$ satisfying the condition (ii) (d) as in Proposition 2.1 is not extremal.

\section{Vector bundles on K3 surfaces}

In this section, we summarize essential properties of line bundles and curves on K3 surfaces, which will be fundamental to our study of the gonality sequence of curves. First of all, we recall the basic numerical invariants. For any vector bundle $E$ on a K3 surface $X$, the Euler number $\chi(E)$ is given by the Riemann-Roch formula:

$$\chi(E)=2\rk(E)+\frac{c_1(E)^2}{2}-c_2(E),$$
where $\chi(E)=h^0(E)-h^1(E)+h^2(E)$. By the Serre duality, $h^i(E) = h^{2-i}(E^{\vee})$ holds for $0 \leq i \leq 2$. In the case of a divisor $D$ on $X$, this implies $h^0(\mathcal{O}_X(D)) + h^0(\mathcal{O}_X(-D)) \geq \chi(\mathcal{O}_X(D)) = 2+D^2/2$. Consequently, any divisor satisfying $D^2 \geq -2$ is either effective or its inverse is effective. 

$\;$

\noindent {\textbf{Remark 3.1}}. If $D$ is a non-zero effective divisor on $X$, its arithmetic genus $P_a(D)$ is related to its self-intersection by the adjunction formula $D^2 = 2P_a(D) - 2$.

$\;$

\noindent The following results, primarily due to Saint-Donat [9, Proposition 2.6], characterize the behavior of linear systems on K3 surfaces.

\begin{prop}.  Let $L$ be a non-trivial line bundle on $X$ whose linear system  $|L|$ is not empty, and has no fixed component. Then:

\smallskip

\smallskip

\noindent {\rm{(i)}} If $L^2>0$, then the general member of $|L|$ is a smooth irreducible curve of genus $1+L^2/2$, and $h^1(L)=0$.

\noindent {\rm{(ii)}} If $L^2 = 0$, then $L \cong \mathcal{O}_X(kF)$ for some elliptic curve $F$ and $k\in\mathbb{N}$, with $h^1(L) = k - 1$. \end{prop}

\smallskip

\smallskip

\noindent A crucial property of K3 surfaces is that for any irreducible curve $C$ with $C^2 \geq 0$, the linear system $|C|$ is base-point-free ([9, Theorem 3.1]). More generally, the following result is known.

\smallskip

\smallskip

\begin{prop} {\rm{{([9, Corollary 3.2])}}}. If $L$ is a non-trivial line bundle on $X$, then $|L|$ has no base point outside of its fixed components whenever $|L|\neq\emptyset$. \end{prop}

\smallskip

\smallskip

\noindent Furthermore, the connectivity of divisors plays a vital role in our analysis. We say that a non-zero effective divisor $D$ on $X$ is {\textit{numerically $m$-connected}} if $D_1. D_2 \geq m$ for any decomposition $D = D_1 + D_2$ into non-zero effective divisors $D_1$ and $D_2$ on $X$.

\smallskip

\smallskip

\begin{prop} {\rm{{([9, Lemma 3.7])}}}. Any member of a base-point-free and big linear system $|L|$ is 2-connected.\end{prop}

\smallskip

\smallskip

\noindent{\textbf{Remark 3.2}}. Note that a base-point-free line bundle on $X$ is nef. By Proposition 3.3, if $L_1$ is a line bundle with $h^0(L_1)\geq2$, and $L_2$ is a base-point-free and big line bundle on $X$, then  we have $L_1.L_2\geq2$.

\section{Lazarsfeld-Mukai bundles on K3 surfaces}

In this section, we review the construction and fundamental properties of Lazarsfeld-Mukai bundles (LM bundles for short), which play a central role in relating the geometry of curves to the moduli of vector bundles on K3 surfaces.

Let $X$ be a K3 surface and $C \subset X$ be a smooth curve of genus $g \geq 2$. Consider a base-point-free line bundle $A$ of degree $d$ on $C$, and let $V \subseteq H^0(A)$ be a subspace of dimension $r+1$ that forms a base-point-free linear system on $C$. The LM bundle $E_{C,(A,V)}$ associated with the data $(C, A, V)$ is defined as the dual of the kernel of the evaluation map $ev: V \otimes \mathcal{O}_X \to A$. Specifically, it fits into the following fundamental exact sequence:
$$0 \longrightarrow V^{\vee} \otimes \mathcal{O}_X \longrightarrow E_{C,(A,V)} \longrightarrow K_C \otimes A^{\vee} \longrightarrow 0. \quad (4.1)$$
Note that $V^{\vee}$ naturally identifies with an $(r+1)$-dimensional subspace of $H^0(E_{C,(A,V)})$. Since the evaluation map is surjective, $E_{C,(A,V)}$ is a locally free sheaf (vector bundle) of rank $r+1$. We simplify the notation to $E_{C,A}$ whenever $V = H^0(A)$. The basic invariants and vanishing theorems for LM bundles are summarized below:

\begin{prop} {\rm{([8, Proposition 2.1])}}. The LM bundle $E = E_{C,(A,V)}$ satisfies the following properties:

\smallskip

\smallskip

\noindent {\rm{(i)}} $\rk(E) = r+1$, $c_1(E) = \mathcal{O}_X(C)$, and $c_2(E) = d$.

\noindent {\rm{(ii)}} $h^1(E) = h^0(A) - r - 1$ and $h^2(E) = 0$.

\noindent {\rm{(iii)}} $E$ is generated by its global sections outside the base locus of $|K_C \otimes A^{\vee}|$.

\noindent {\rm{(iv)}} $\chi(E^{\vee}\otimes E)=2(1-\rho(g,r,d))$, where $\rho(g,r,d) = g - (r+1)(g-d+r)$ is the Brill-Noether number. \end{prop}

\noindent Property (iv) implies that if $\rho(g,r,d) < 0$, then $\chi(E^\vee \otimes E) > 2$, which means $h^0(E^\vee \otimes E) \geq 2$. Thus, $E$ is not a simple bundle, admitting a non-zero endomorphism that is not a multiple of the identity. In the rank two case ($r=1$), such non-simplicity leads to a specific destabilizing structure.

\begin{prop} {\rm{([1, Lemma 2.1], [2, Lemma 4.4])}}. Suppose $\rk (E_{C,(A,V)}) = 2$. If $\rho(g,1,d) < 0$, then there exist line bundles $M$ and $N$ on $X$ such that $h^0(M) \geq 2$, $h^0(N) \geq 2$, and $N$ is base-point-free. Then $E_{C,(A,V)}$ is described by one of the following:

\smallskip

\smallskip

\noindent {\rm{(i)}} $E_{C,(A,V)} \cong M \oplus N$.

\noindent {\rm{(ii)}} An extension $0 \to M \to E_{C,(A,V)} \to N \otimes \mathcal{J}_W \to 0$, where $W \subset X$ is a zero-dimensional subscheme. \end{prop}

\noindent{\textbf{Remark 4.1}}. In Proposition 4.2, if the case (ii) occurs, then $N\subset M$. Hence, we may assume that $M^2 \geq N^2$ by swapping $M$ and $N$ if necessary. By Proposition 4.2, we observe $h^0(E_{C,(A,V)} \otimes M^\vee) > 0$. By tensoring $(4.1)$ with $M^\vee$, we obtain the exact sequence
$$0 \longrightarrow V^{\vee} \otimes M^{\vee} \longrightarrow E_{C,(A,V)} \otimes M^{\vee} \longrightarrow N|_C \otimes A^{\vee} \longrightarrow 0.$$
Hence, $A$ is contained in $N|_C$ (as $V^\vee \otimes M^\vee$ has no sections if $h^0(M)>0$). In particular, $N$ must be a non-trivial line bundle.

$\;$

\noindent The structural decomposition of the bundle $E_{C,(A,V)}$ provided in Proposition 4.2 will be essential in Section 5, where we relate the second gonality of aCM curves to the intersection numbers of the line bundles $M$ and $N$.

\section{Proof of Theorem 1.1}

In this section, we prove Theorem 1.1. Let $X$ be a smooth quartic surface in $\mathbb{P}^3$, and let $H$ be the line bundle associated with a hyperplane section of $X$. Let $C\subset X$ be a smooth aCM curve of genus $g\geq9$ and degree $d$ with respect to $H|_C$, and let $A$ be a linear series $g_{d_2}^2$ on $C$ with $g>2d_2-2$ which computes $\gamma_C$. Since $h^0(H|_C(-P))=3$, by Remark 2.1, for any point $P\in C$, we obtain $d_2\leq d-1$. Since $H|_C$ is very ample, if $A\subset |H|_C|$, then we have $d_2=d-1$. We show the following proposition to prove Theorem 1.1.

\begin{prop} Let the notations be as above. Then $A\subset |H|_C|$.\end{prop}

\noindent Let $V\subset H^0(A)$ be a general subspace of dimension two. Then it forms a base-point-free pencil $|V|$ on $C$. Since $\rho(g,1,d_2)<0$, by Proposition 4.2, there exist two line bundles $M$ and $N$ on $X$ with $h^0(M)\geq2$, $h^0(N)\geq2$, such that $N$ is base-point-free, and
$$E_{C,(A,V)}\cong M\oplus N,\quad (5.1)$$
or there exists a zero-dimensional subscheme $W\subset X$ such that $E_{C,(A,V)}$ sits in the exact sequence
$$0\longrightarrow M \longrightarrow E_{C,(A,V)}\longrightarrow N\otimes\mathcal{J}_W \longrightarrow0.\quad (5.2)$$
Hence, we have
$$N.C-N^2=M.N\leq d_2\leq d-1.\quad (5.3)$$

\smallskip

\smallskip

\noindent By Remark 4.1, it is sufficient to show the following inclusion to prove Proposition 5.1.
$$N\subset H.\quad (5.4)$$
Indeed, since $g\geq9$, we have $h^0(\mathcal{O}_X(1)(-C))=0$. Since $h^0(N^{\vee}(1)(-C))=0$, if $N\subset H$, then, by the exact sequence
$$0\longrightarrow N^{\vee}(1)(-C)\longrightarrow N^{\vee}(1)\longrightarrow N^{\vee}(1)|_C\longrightarrow0,$$
we have $h^0(N^{\vee}(1)|_C)\geq h^0(N^{\vee}(1))>0$. Since $N|_C\subset H|_C$, by Remark 4.1, we obtain $A\subset H|_C$. By Remark 2.2, there exist a divisor $D$ on $X$ and $s\in\mathbb{N}$ such that $\mathcal{O}_X(D)$ is an initialized and aCM line bundle on $X$, and $C\sim sH+D$. Hence, by the inequality (5.3), we have
$$s(N.H-4)\leq D.(H-N)+N^2-1.\quad (5.5)$$
We prepare several lemmas to prove the inclusion (5.4), for each case as in Proposition 2.1 (ii).

\newtheorem{lem}{Lemma}[section]

\begin{lem} Keeping above notations, the following statements hold.

\smallskip

\smallskip

\noindent {\rm{(i)}} If $N^2=0$, then there exists an elliptic curve $F$ on $X$ such that $F\in|N|$.

\smallskip

\noindent {\rm{(ii)}} $N\subset M$.

\smallskip

\noindent {\rm{(iii)}} If $N.H\leq 4$, then $N\subset H$.

\smallskip

\noindent {\rm{(iv)}} If $N.H\geq5$, then $s\geq2$.

\smallskip

\noindent {\rm{(v)}} If $N^2=2$ and $h^1(M)=0$, then $A=|N|_C|$. 

\smallskip

\noindent {\rm{(vi)}} If $N^2=4$, then $N.H\neq5$.

\smallskip

\noindent {\rm{(vii)}} $(M-N)^2\geq 4s^2+2(D.H-8)s+D^2-4D.H+4$. \end{lem}

{\textit{Proof}}. (i) Assume that $N^2=0$. Then, by Proposition 3.1 (ii), there exist an elliptic curve $F$ and $t\in\mathbb{N}$ such that $N\sim tF$. By the inequality (5.3), we have $N.C=M.N\leq d_2$. By Remark 4.1, we have $A\cong N|_C$. Since $h^0(A)=3$ and $h^0(M^{\vee})=0$, by the exact sequence 
$$0\longrightarrow M^{\vee}\longrightarrow N\longrightarrow N|_C\longrightarrow0,\quad (5.6)$$
we have $t=1$ or 2. We show that $t=1$. First of all, by the exact sequence
$$0\longrightarrow \mathcal{O}_X(F-C)\longrightarrow \mathcal{O}_X(F)\longrightarrow \mathcal{O}_C(F)\longrightarrow0,$$
we have $h^0(\mathcal{O}_C(F))\geq2$. Since $2d_2-2<g$, we have $F.C=d_2/t\leq d_2\leq g-1$. Hence, $\mathcal{O}_C(F)$ contributes to $\gamma_C$, and we obtain 
$$\gamma(\mathcal{O}_C(F))=F.C-2(h^0(\mathcal{O}_C(F))-1)\leq F.C-2=\frac{1}{t}N.C-2.$$
Since $F.C\geq F.H\geq 3$, if $t=2$, then $\gamma_C=N.C-4=2F.C-4\geq2$. Hence,
$$\frac{1}{t}N.C-2=\frac{1}{2}\gamma_C<\gamma_C.$$
This contradicts the minimality of $\gamma_C$. Therefore, $t=1$.

\smallskip

\smallskip

(ii) Assume that $N\nsubseteq M$. By Remark 4.1, we obtain the isomorphism (5.1). Hence, $M$ and $N$ is base-point-free. Since $\rho(g,1,d_2)<0$ and $g\geq9$, if $N^2=0$, we have $M^2>0$. In fact, if $M^2=0$, then we have $2g-2=C^2=(M+N)^2=2d_2<g+2$. This case does not occur. Hence, by Proposition 3.1 (i), we have $h^1(M)=0$. By the above argument, this implies that $h^1(E_{C,(A,V)})=0$. This contradicts the assertion of Proposition 4.1 (ii).

By the assertion of (i), if $N^2>0$, then we may assume that $M^2>0$. By Proposition 3.1 (i) and the same reason as above, we obtain the contradiction $h^1(E_{C,(A,V)})=0$.

\smallskip

\smallskip

(iii) Since $N^2\geq0$ and $H$ is not hyperelliptic, by Remark 3.2, $N.H\geq3$. Assume that $N.H\leq 4$. Then, by the Hodge index theorem, $N^2\leq4$. 

We consider the case where $N^2=0$. If $N.H=3$, then $(H-N)^2=-2$ and $H.(H-N)=1$. Hence, $N\subset H$. Assume that $N.H=4$. Since $(2H-N)^2=0$ and $H.(2H-N)=4$, by the Riemann-Roch theorem, we have $h^0(2H-N)>0$. Moreover, since $|N|$ is base-point-free, $|2H-N|$ is also base-point-free. In fact, if $|2H-N|$ is not base-point-free, then, by Proposition 3.2, it has a fixed component $\Gamma$. We denote by $N_0$ the movable part of $|2H-N|$. Since $H$ is not hyperelliptic, $N_0.H=3$ and $\Gamma.H=1$. Since $\Gamma$ is a line on $X$ and the general member of $|N_0|$ is irreducible, by Proposition 3.1, we have $h^1(\mathcal{O}_X(N_0))=0$. Since, by Proposition 2.1, $2H-N$ is an aCM line bundle on $X$, we obtain $h^1(2H-N)=0$. Hence, by the Riemann-Roch theorem, we get
$$\chi(N_0)=h^0(N_0)=h^0(2H-N)=\chi(2H-N).$$
This means that $N_0^2=0$. Since $2H-N\sim\Gamma+N_0$, we have $N\sim (H-\Gamma)+(H-N_0)$. Since $(H-\Gamma)^2=0$, $(H-N_0)^2=-2$, and $(H-\Gamma).(H-N_0)=1$, the member of $|H-N_0|$ is the fixed component of $|N|$. This is a contradiction. Hence, if we take a smooth irreducible member $F_0\in |2H-N|$, then $h^1(M|_{F_0})=h^0(M^{\vee}|_{F_0})=0$. Since $\mathcal{O}_X(D)$ is aCM, $h^1(\mathcal{O}_X(D)(s-2))=0$. By the exact sequence
$$0\longrightarrow\mathcal{O}_X(D)(s-2)\longrightarrow M\longrightarrow M|_{F_0}\longrightarrow0,$$
we have $h^1(M)=0$. By the assertion of (i), $|N|$ is a pencil. Hence, by the exact sequence (5.6), $|N|_C|$ is also a pencil. Since $A\cong N|_C$, this is a contradiction.

We consider the case where $N^2=2$. If $N.H=3$, then $(H-N)^2=0$ and $H.(H-N)=1$. However, since the movable part of $|H-N|$ is not empty, by Remark 3.2, this case does not occur. If $N.H=4$, we have $(H-N)^2=-2$, and hence, $h^0(N-H)>0$ or $h^0(H-N)>0$. Since $H.(N-H)=0$ and $N-H$ is not trivial, this contradicts the ampleness of $H$. 

We consider the case where $N^2=4$. By the Hodge index theorem, we have $N.H\geq4$. If $N.H=4$, then $(N-H)^2=0$ and $H.(N-H)=0$. Since $H$ is ample, we have $N\sim H$. By the above observation, we obtain the assertion.

\smallskip

\smallskip

(iv) Assume that $N.H\geq 5$ and $s=1$. Since $H$ is ample, by the assertion of (ii), we obtain
$$H.(H+D-2N)=H.(C-2N)=H.(M-N)\geq0.$$
Since, by Proposition 2.1, $H.D\leq6$, we have
$$10\geq 4+H.D\geq 2N.H\geq 10.$$
This means that $H.D=6$, $D^2=4$, and $N.H=5$. Since $H.(H+D-2N)=0$, by the ampleness of $H$, we have $H+D\sim 2N$, and hence, $N^2=5$. By Remark 3.1, this case does not occur.

\smallskip

\smallskip

(v) Since $N^2=2$, we have $h^0(N)=3$. Since $h^1(M)=0$ and $h^0(M^{\vee})=0$, by the exact sequence (5.6), $N|_C$ is base-point-free, and $h^0(N|_C)=3$. By Remark 4.1 and the hypothesis on $A$, we obtain $A=|N|_C|$.

\smallskip

\smallskip

(vi) If $N^2=4$ and $N.H=5$, then $(N-H)^2=-2$ and $H.(N-H)=1$. By the ampleness of $H$, we have $h^0(N-H)>0$. However, since $N$ is base-point-free, $h^0(N)=4$ and $h^0(H)=4$, we have a contradiction.

\smallskip

\smallskip

(vii) The inequality (5.5) implies that $N.C\leq H.D+N^2-1+4s$. Since $(M-N)^2=(C-2N)^2$, the assertion follows immediately. $\hfill\square$

$\;$

\noindent {\textbf{Remark 5.1}}. If $N^2\geq6$, then, by the inequality (2.1), $N.H\geq6$. In particular, if $N.H=6$, then $N^2=6$.

\begin{lem} If $N^2\geq6$, then 
$$3-\frac{H.D}{4}\leq s\leq\frac{1}{16}\{(H.D-2)^2+60-4D^2\}.$$\end{lem}

{\textit{Proof}}. Since $C\sim M+N$, we have $2M.N=C^2-(M^2+N^2)$. By the Hodge index theorem, $(H.M)^2\geq 4M^2$ and $(H.N)^2\geq4N^2$, and hence, we obtain
$$\frac{(M.H)^2+(N.H)^2}{4}\geq M^2+N^2.$$
If we set $d_M:=M.H$ and $d_N:=N.H$, we get
$$2M.N\geq C^2-\frac{d_M^2+d_N^2}{4}.$$
By Remark 5.1, we have $d_N\geq6$. Moreover, since $d_M+d_N=H.C=4s+H.D$, by Lemma 5.1 (ii), we have $d_N\leq 2s+H.D/2$. Hence, $s\geq 3-H.D/4$. If we set
$$f(d_N)=\frac{1}{4}\{(4s+H.D-d_N)^2+d_N^2\},$$
then $f(d_N)\leq f(6)$, and hence,
$$2M.N\geq C^2-f(d_N)\geq C^2-f(6)=12s+D^2-9-\frac{(H.D-6)^2}{4}.$$
On the other hand, by the inequality (5.3), we have $M.N\leq 4s+H.D-1$. Therefore, we obtain the assertion. $\hfill\square$

\smallskip

\smallskip

\noindent In order to prove Proposition 5.1, we separate it into several cases as in Proposition 2.1 (ii).

\smallskip

\smallskip

\begin{prop} If $D^2=-2$, then $N\subset H$.\end{prop}

{\textit{Proof}}. By Lemma 5.1 (iii), it is sufficient to show that $N.H\leq4$. Assume that $N.H\geq5$. By Lemma 5.1 (iv), we have $s\geq2$. Since $N$ is base-point-free, $N.D\geq0$. By Proposition 2.1, we have $H.D\leq3$. 

We consider the case where $N^2=0$. By the inequality (5.5), we have
$$s(N.H-4)\leq D.(H-N)-1.$$
Hence, we obtain $s=2$, $N.H=5$, $N.D=0$, and $H.D=3$. Thus, $H.(M-N)=1$. By Lemma 5.1(ii), the member of $|M-N|$ is a line on $X$. However, since $(M-N)^2=C^2-4N.C=-14$, this case does not occur.

We consider the case where $N^2=2$. Similarly, by the inequality (5.5), we have
$$s(N.H-4)\leq D.(H-N)+1.\quad (5.7)$$
Since $D.(H-N)+1\leq4$ and $s\geq2$, if $N.H\geq6$, by the inequality (5.7), we obtain $N.H=6$, $H.D=3$, $D.N=0$, $d_2=M.N=N.C-2$, and $s=2$. This means that $\rho(g,1,d_2)=2N.C-g-6=4>0$. However, this contradicts the hypothesis of Theorem 1.1. Since $N.H=5$, we have $D.N\leq H.D+1-s\leq2$. If $D.N=2$, then $D.(H-N)+1\leq2$. By the inequality (5.7), we have $s=2$, $H.D=3$, $D.N=2$, and $d_2=M.N=N.C-2$. Thus, $\rho(g,1,d_2)=2N.C-g-6=4>0$. This also is a contradiction. Hence, $D.N\leq1$. If $D.N=1$, by the inequality (5.7), $H.D\geq2$. Thus, $2\leq H.(N-D)\leq3$. Since $(N-D)^2=-2$, by Proposition 2.1, $M$ is aCM and hence, $h^1(M)=0$. By Lemma 5.1 (v), $A=|N|_C|$. We obtain $d_2=N.C=5s+1$. By the inequality (5.3), we have $5s+1\leq d-1\leq 4s+2$. This contradicts the fact that $s\geq2$. Thus, we have $D.N=0$. In this case, $(N-D)^2=0$ and $2\leq H.(N-D)\leq4$. Since $H$ is not hyperelliptic, $H.(N-D)=3$ or 4. By Proposition 2.1, $N-D$ is aCM and hence, $M$ is also aCM. Thus, $h^1(M)=0$. By Lemma 5.1 (v), we obtain $A=|N|_C|$. Hence, $d_2=N.C=5s$. On the other hand, since $1\leq H.D\leq 2$, by the inequality (5.3), we have $5s\leq d-1\leq 4s+1$. This also contradicts the fact that $s\geq2$. 

We consider the case where $N^2=4$. By Lemma 5.1 (vi), $N.H\geq6$. Since $H.D\leq3$, by the ampleness of $H$, we have
$$0\leq H.(M-N)=H.(sH+D-2N)\leq 4s-9.$$
Thus, $s\geq3$. By the inequality (5.5), we have
$$s(N.H-4)\leq D.(H-N)+3,$$
and hence, $N.H=6$, $H.D=3$, and $D.N=0$. In this case, we obtain $(H+D-N)^2=0$, and $H.(H+D-N)=1$. However, this contradicts Remark 3.2.

We consider the case where $N^2\geq6$. By Lemma 5.2, we have $3\leq s\leq4$. By the inequality (5.5),
$$D.N\leq H.D+N^2-1-s(N.H-4),\quad (5.8)$$
and hence, we have 
$$(H+D-N)^2\geq 2(s-1)N.H-N^2-8s+4.\quad (5.9)$$
By Lemma 5.1 (ii) and the ampleness of $H$, we obtain
$$4s+H.D-2N.H=H.(C-2N)=H.(M-N)\geq0.\quad (5.10)$$
By Remark 5.1, we have $6\leq N.H\leq9$. Assume that $N.H=9$. Then we have $H.(M-N)\leq1$. If $H.(M-N)=0$, then $s=4$ and $H.D=2$. By the ampleness of $H$, we have $C\sim 2N$. This means that $4N^2=C^2=78$. This case does not occur. Since $H.(M-N)=1$, we have $s=4$, $H.D=3$, and the member of $|M-N|$ is a line on $X$. However, by Lemma 5.1 (vii), we have the contradiction $(M-N)^2\geq 4H.D+2=14$. Hence, $N.H\leq8$. 

Assume that $s=4$. If $N.H=8$, by the Hodge index theorem, we have $N^2\leq 16$. By the inequality (5.9), we have $(H+D-N)^2\geq 20-N^2\geq4$ and $1\leq H.(N-H-D)\leq3$. By the Hodge index theorem, this is a contradiction. If $N.H=7$, then $N^2\leq 12$. By the inequality (5.9), $(H+D-N)^2\geq 14-N^2\geq2$ and $0\leq H.(N-H-D)\leq2$. However, this case also does not occur. Since $N.H=6$, by Remark 5.1, $N^2=6$. By the inequality (5.9), $(H+D-N)^2\geq2$ and $-1\leq H.(N-H-D)\leq 1$. By the same reason as above, this is the contradiction. 

Assume that $s=3$. By the inequality (5.10), $N.H\leq 7$. Assume that $N.H=7$. Then $H.(M-N)\leq1$. If $H.(M-N)=0$, then $H.D=2$. We have $C\sim 2N$, and hence, $4N^2=C^2=46$. However, this case does not occur. Since $H.(M-N)=1$, $H.D=3$ and the member of $|M-N|$ is a line on $X$. On the other hand,
$$(M-N)^2=(C-2N)^2=C^2-4N.C+4N^2=52-4N.C+4N^2\equiv 0 \mod 4.$$
This is a contradiction. We obtain $N.H=6$. By Remark 5.1, $N^2=6$. First of all, we have
$$H.(M-N)=H.D,\quad (5.11)$$
and
$$(M-N)^2=6H.D-14-4N.D.\quad (5.12)$$
If $H.D=1$, then, by the equality (5.11), the member of $|M-N|$ is a line on $X$. However, since $(M-N)^2=-8-4N.D$, this case does not occur. If $H.D=2$, by the equality (5.11), one of the following cases occurs.

\smallskip

\smallskip

\noindent (i) There exist two lines $\Gamma_1$ and $\Gamma_2$ on $X$ with $\Gamma_1.\Gamma_2=0$ or $-2$ such that $\Gamma_1+\Gamma_2\in |M-N|$.

\noindent (ii) There exists a conic $\Gamma$ on $X$ such that $\Gamma\in |M-N|$ (not necessarily irreducible).

\noindent By the equality (5.12), $(M-N)^2=-2-4N.D=-2$. Hence, $N.D=0$. Since $(N-D)^2=H.(N-D)=4$, by the ampleness of $H$, we have $N-D\sim H$. Since $N\sim H+D$, we have $M\sim 2H$. On the other hand, by the Riemann-Roch theorem, we have
$$h^0(E_{C,(A,V)})=\chi(E_{C,(A,V)})+h^1(E_{C,(A,V)})=g-d_2+4=28-d_2,$$
$h^0(N)=5$, and $h^0(M)=10$. By the exact sequence (5.2), we have $28-d_2\leq 15$. Since $d_2\geq 13=d-1$ and $N$ is base-point-free, the length of $W$ is zero. Since $h^1(N)=h^1(M)=0$, this contradicts the assertion of Proposition 4.1 (ii). Thus $H.D=3$. By the inequality (5.9), $2-2N.D=(H+D-N)^2\geq-2$ and $H.(H+D-N)=1$. Hence, the member of $|H+D-N|$ is a line on $X$ and $N.D=2$. Since $d_2=M.N$, in the exact sequence (5.2), the length of $W$ is zero. By Proposition 2.1, $M$ is aCM and hence, $h^1(M)=0$. Since $h^1(N)=0$, we have the contradiction $h^1(E_{C,(A,V)})=0$, by Proposition 4.1. By the above observation, we obtain the consequence of Proposition 5.2. $\hfill\square$

$\;$

\begin{prop} If $D^2=0$ and $D\neq0$, then $N\subset H$.\end{prop}

{\textit{Proof}}. By Proposition 2.1, $3\leq H.D\leq4$. By Lemma 5.1 (iii), it is sufficient to show that $N.H\leq 4$. Assume that $N.H\geq 5$. We note that, by Lemma 5.1 (iv), $s\geq2$.

We consider the case where $N^2=0$. Since $N.D\geq0$, by the inequality (5.5), $N.H=5$ and hence,  $N.D\leq1$. Since $(N-D)^2\geq-2$ and $H.(N-D)\geq1$, we have $h^0(N-D)>0$. However, since, by Lemma 5.1 (i), $|N|$ is a base-point-free pencil on $X$ and $h^0(\mathcal{O}_X(D))=2$, this case does not occur.

We consider the case where $N^2=2$. By Remark 3.2, $N.D\geq2$. By the inequality (5.5), $N.H=5$ and 
$$(s, H.D, N.D)=(3,4,2), (2,4,2), (2,4,3),\text{or } (2,3,2).$$
If $(s, H.D, N.D)=(2,4,2)$, then  $(N-D)^2=-2$ and $H.(N-D)=1$. Since, by Proposition 2.1, $M$ is aCM, we have $h^1(M)=0$. By Lemma 5.1 (v), $A=|N|_C|$ and hence, $d_2=N.C$. This implies the contradiction $\rho(g,1,d_2)=2N.C-g-2=5>0$. Assume that $(s, H.D, N.D)\neq(2,4,2)$. Then, by the inequality (5.5), $d_2=M.N$. If $s=3$, we have $(N-D)^2=-2$ and $H.(N-D)=1$. Hence, by Proposition 2.1, $M$ is aCM. Since $h^1(M)=0$, by Lemma 5.1 (v), $d_2=N.C$. This is a contradiction. If $s=2$, then we have the contradiction $\rho(g,1,d_2)=2N.C-g-6=5+2D.(N-H)=3>0$. 

We consider the case where $N^2=4$. By Lemma 5.1 (vi), $N.H\geq6$. If $s=2$, by Lemma 5.1 (ii), $H.(M-N)=0$, and hence, we have $C\sim 2N$. Since $(2H+D)^2=C^2=4N^2$, we have the contradiction $H.D=0$. Thus, $s\geq3$. Since $N$ is base-point-free and big, by Remark 3.2, we have $N.D\geq2$. However, by the inequality (5.5), we have the following contradiction.
$$6\leq s(N.H-4)\leq D.(H-N)+3\leq5.$$

We consider the case where $N^2\geq6$. By Lemma 5.2, $s\leq4$. Assume that $s=4$. By Lemma 5.2, $H.D=4$. By Lemma 5.1 (vii), we have $(M-N)^2\geq 20$.
By the Hodge index theorem, we have
$$20-2N.H=H.(C-2N)=H.(M-N)\geq9,$$
and hence, $N.H\leq5$. However, this contradicts Remark 5.1. Thus, $s\leq3$. 

Assume that $s=3$. If $H.D=4$, then, by Lemma 5.1 (vii), $(M-N)^2\geq0$. Since $H$ is not hyperelliptic, $16-2N.H=H.(M-N)\geq3$. By Remark 5.1, we have $N.H=N^2=6$. By the inequality (5.8), $(N-D)^2\geq0$. However, since $H$ is not hyperelliptic and $H.(N-D)=2$, we have the contradiction. Hence, $H.D=3$. Since $d=H.C=15$, we have $15-2N.H=H.(M-N)\geq0$. By Remark 5.1, $6\leq N.H\leq 7$. 

Assume that $N.H=7$. Since $H.(M-N)=1$, we have $M^2-N^2=C.(M-N)=3+D.(M-N)$. Since $M^2-N^2\equiv 0 \mod 2$, we get
$$D.(M-N)\equiv 1\mod 2.\quad (5.13)$$
On the other hand, since $H.D=3$ and $H.(M-N)=1$, the members of $|H-D|$ and $|M-N|$ are lines on $X$. Hence, 
$$1-D.(M-N)=(H-D).(M-N)=0,1,\text{ or }-2.$$
By the equality (5.13), $D.(M-N)=1$ or 3. If $D.(M-N)=1$, then $M^2=N^2+4$. Since $(M-N)^2=-2$ and $(M+N)^2=C^2=54$, we have $M.N=14$. Hence, we have the contradiction $N^2=11$. If $D.(M-N)=3$, then $M-N\sim H-D$. Hence, $M\sim 2H$ and $N\sim H+D$. In this case, since $M.N=14=d-1$, we have $M.N=d_2$. Hence, in the exact sequence (5.2), the length of $W$ is zero. Since $h^1(M)=h^1(N)=0$, we obtain the contradiction $h^1(E_{C,(A,V)})=0$, by Proposition 4.1 (ii). Hence, $N.H=6$. By Remark 5.1, we have $N^2=6$. By the inequality (5.8), we have $N.D\leq2$. Hence, $(H+D-N)^2=4-2N.D\geq0$. However, by Remark 3.2, we have the contradiction $H.(H+D-N)=1$. Thus, $s=2$. By Lemma 5.2,  $H.D=4$. Since $d=12$, we have $N.H=6$. By Remark 5.1, $N^2=6$ and $H.(M-N)=0$. By Lemma 5.1 (ii) and the ampleness of $H$, we have $C\sim 2N$. However, we have the contradiction $32=C^2=4N^2=24$. $\hfill\square$

\begin{prop} If $D=0$, then $N\subset H$.\end{prop}

{\textit{Proof}}. By Lemma 5.1 (iii), we prove $N.H\leq 4$. Assume that $N.H\geq5$. By the inequality (5.5), we have
$$s(N.H-4)\leq N^2-1.\quad (5.14)$$
Since $s\geq2$, we have $N^2\geq4$. If $N^2=4$, then, by Lemma 5.1 (vi), we have $N.H\geq6$. This contradicts the inequality (5.14). Hence, $N^2\geq6$. By Lemma 5.2, $3\leq s\leq4$. If $s=4$, by Lemma 5.1 (vii), we have $(M-N)^2\geq4$. By the Hodge index theorem, we obtain
$$16-2N.H=H.(4H-2N)=H.(M-N)\geq4.$$
By Remark 5.1, we have $N.H=6$ and $N^2=6$. However, this also contradicts the inequality (5.14). Assume that $s=3$. Since $d=3H^2=12$, we have $N.H=6$, and hence, $N^2=6$. By the same reason as above, this is a contradiction. $\hfill\square$

\begin{prop} If $D^2=2$, then $N\subset H$.\end{prop}

{\textit{Proof}}. By Lemma 5.1 (iii), we show that $N.H\leq 4$. Assume that $N.H\geq5$.

We consider the case where $N^2=0$. Since $s\geq2$ and, by Proposition 2.1, $H.D=5$, we have $D.N\leq2$ by the inequality (5.5). Thus $(N-D)^2\geq-2$ and $H.(N-D)\geq0$. This means that $h^0(N-D)>0$. By Lemma 5.1 (i), $|N|$ is a pencil. This is a contradiction.

We consider the case where $N^2=2$. By the inequality (5.5), $D.N\leq4$. If $D.N\leq3$, we have $(N-D)^2=4-2N.D\geq-2$, and $H.(N-D)\geq0$. Hence, by the ampleness of $H$, we have $h^0(N-D)>0$. Since $\mathcal{O}_X(D)$ is aCM, we obtain $h^0(\mathcal{O}_X(D))=3$. Since $|N|$ is base-point-free net, $N\sim D$. Since $M\sim sH$, we have $h^1(M)=0$. By Lemma 5.1 (v), we have $A=|N|_C|$. Since $5s+2=N.C=d_2\leq d-1=4s+4$, we have $s=2$. Since $d_2=12$ and $g=20$, we obtain the contradiction $\rho(g,1,d_2)=2d_2-g-2=2>0$. Hence, $D.N=4$, $s=2$, $N.H=5$, and $d_2=M.N=N.C-2$. However, we have the contradiction $\rho(g,1,d_2)=2N.C-g-6=2>0$.

We consider the case where $N^2=4$. By Lemma 5.1 (vi), we have $N.H\geq6$. By the Hodge index theorem, we have $D.N\geq3$. If $D.N=3$, we have $(N-D)^2=0$. Since $s\geq2$, by the inequality (5.5), $N.H=6$. Thus, $H.(N-D)=1$. This contradicts Remark 3.2. Hence, $D.N\geq4$. By using the inequality (5.5) again, we obtain $s=2$, $N.H=6$, $D.N=4$, and $d_2=M.N=N.C-4$. Hence, we obtain the contradiction $\rho(g,1,d_2)=2N.C-g-10=2>0$.

We consider the case where $N^2\geq6$. By Lemma 5.2, $s\leq 3$. If $s=3$, then, by Lemma 5.1 (vii), $(M-N)^2\geq4$. By the Hodge index theorem, $H.(M-N)\geq4$. Since $d=C.H=17$, we obtain $N.H\leq6$. By Remark 5.1, $N.H=N^2=6$. By the Hodge index theorem, $D.N\geq4$. By the inequality (5.5), $D.N=4$. This means that $(N-D)^2=0$ and $H.(N-D)=1$. However, this contradicts Remark 3.2. Hence, $s=2$. Since $d=C.H=13$, by Lemma 5.1 (ii), $N.H\leq6$. By Remark 5.1, $N.H=N^2=6$. Since $H.(M-N)=1$, the member of $|M-N|$ is a line on $X$, and hence, $(2H+D-2N)^2=(C-2N)^2=(M-N)^2=-2$. We obtain $D.N=4$. Since $(N-D)^2=0$ and $H.(N-D)=1$, by Remark 3.2, we have a contradiction. $\hfill\square$

\begin{prop} If $D^2=4$, then $N\subset H$.\end{prop}

{\textit{Proof}}. By Lemma 5.1 (iii), we show that $N.H\leq4$. Assume that $N.H\geq5$. 

We consider the case where $N^2=0$. Since $H.D=6$ and $s\geq2$, by the inequality (5.5), $D.N\leq3$. Then we have $(D-N)^2=4-2D.N\geq-2$. Since $h^0(\mathcal{O}_X(D))=4$ and, by Lemma 5.1 (i), $|N|$ is a pencil, $h^0(D-N)>0$. Since $D^2\neq0$ and $H.(D-N)\leq1$, by the ampleness of $H$, we have $H.(D-N)=1$. Hence, the member of $|D-N|$ is a line on $X$. Thus, $N.D=3$. By using the inequality (5.5) again, we obtain $s=2$, $N.H=5$, and $d_2=M.N$. By Proposition 2.1, $D-N$ is aCM, and hence, $M$ is also aCM. Since the length of $W$ is zero, and $h^1(M)=h^1(N)=0$, by the exact sequence (5.2), we have $h^1(E_{C,(A,V)})=0$. This contradicts the assertion of Proposition 4.1 (ii).

We consider the case where $N^2=2$. By the Hodge index theorem, $D.N\geq3$. Since $s\geq2$, if $N.H\geq6$, then, by the inequality (5.5), $s=2$, $N.H=6$, and $D.N=3$. Then $(N-D)^2=0$ and $H.(N-D)=0$. Since $H$ is ample, $N\sim D$. However, since $D^2=4$, this case does not occur. Thus, $N.H=5$. Then we obtain $D.N\geq5$. In fact, if $D.N\leq4$, we have $(D-N)^2=6-2N.D\geq-2$ and $H.(D-N)=1$. This means that the member of $|D-N|$ is a line on $X$, and hence, $M$ is aCM and $N.D=4$. Since $h^1(M)=0$, by Lemma 5.1 (v), we get $A=|N|_C|$. Hence, $5s+4=N.(sH+D)=N.C=d_2\leq d-1=4s+5$. This contradicts the fact that $s\geq2$. By the inequality (5.5), $s=2$, $D.N=5$, and $d_2=M.N$. This implies the contradiction $\rho(g,1,d_2)=2N.C-g-6=1>0$. 

We consider the case where $N^2=4$. By the Hodge index theorem, $D.N\geq4$. By Lemma 5.1 (vi) and the inequality (5.5), $N.H=6$, and hence, $H.(D-N)=0$. If $D.N=4$, then $(D-N)^2=0$. By the ampleness of $H$, we have $N\sim D$. Since $M\sim sH$, we have $6s=M.N\leq 4s+D.H-1=4s+5$, and hence, $s=2$. We get $M-N\sim 2H-D$. Since $\mathcal{O}_X(D)$ is aCM, $M-N$ is also aCM. By Lemma 5.1 (ii), $h^0(M-N)>0$. However, since $(M-N)^2=-4$, we have the contradiction $h^1(M-N)\neq0$. Thus, $D.N\geq5$. By the inequality (5.5), $s=2$, $D.N=5$, and $d_2=M.N=N.C-4$. We obtain the contradiction $\rho(g,1,d_2)=2N.C-g-10=1>0$.

We consider the case where $N^2\geq6$. By Lemma 5.2, $s\leq3$. If $s=3$, by Lemma 5.1 (vii), $(M-N)^2\geq8$. By the Hodge index theorem,
$$18-2N.H=H.(M-N)\geq6.$$
Hence, $N.H\leq6$. By Remark 5.1, we have $N.H=N^2=6$. By using the Hodge index theorem again, $D.N\geq5$. By the inequality (5.5), $D.N=5$. Therefore, $(N-D)^2=0$ and $H.(N-D)=0$. By the ampleness of $H$, we have the contradiction $N\sim D$. Thus, $s=2$. Since $d=H.C=14$, by Lemma 5.1 (ii) and Remark 5.1, $6\leq N.H\leq7$. If $N.H=7$, by the ampleness of $H$, we have $C\sim 2N$. This means the contradiction $N^2=11$. Hence, $N.H=N^2=6$. By the inequality (5.5) and the Hodge index theorem, $5\leq D.N\leq 7$. If $D.N\leq6$, then $(N-D)^2\geq-2$ and $H.(N-D)=0$. Hence, we obtain the contradiction $N\sim D$. If $D.N=7$, we get $d_2=M.N=N.C-6$. Therefore, we have the contradiction $\rho(g,1,d_2)=2N.C-g-14=1>0$. $\hfill\square$

$\;$

\noindent Finally, we show that the hypothesis $g>2d_2-2$ as in Theorem 1.1 cannot be avoided, by giving the following example.

$\;$

\noindent{\textbf{Example 5.1}}. By [7, Theorem 4.1], there exist a smooth quartic surface $X$ in $\mathbb{P}^3$, and a smooth rational curve $D$ of degree 3 on $X$ such that $\Pic(X)$ is generated by the classes of $D$ and a hyperplane section $H$ of $X$. Since $(2H-D)^2=2$ and $H.(2H-D)=5$, by the proof of Proposition 2.1, the general member of $|2H-D|$ is a smooth aCM curve of genus 2. We take a smooth irreducible curve $C\in|3H-D|$. Then the genus $g$ of $C$ is 9, $|\mathcal{O}_C(2H-D)|=g_7^2$, and $\gamma_C=3$. Indeed, $\gamma_C\leq\gamma(\mathcal{O}_C(2H-D))=3<4=\lfloor (g-1)/2\rfloor$. There exists a line bundle $L\in\Pic(X)$ such that $h^0(L)\geq2$, $h^0(L^{\vee}(C))\geq2$, and $\gamma(L|_C)=\gamma_C$ ([5, Theorem]). We set $L\sim sH+tD$, for $s,t\in\mathbb{Z}$. Since $h^0(\mathcal{O}_X(mD))\leq1$ for any $m\in \mathbb{Z}$, we have $1\leq s\leq2$. Since $h^1(L)=h^1(L^{\vee}(C))=0$ ([6, Proposition 2.6]), we obtain $L^2\geq0$ and $(C-L)^2\geq0$, by the Riemann-Roch theorem. By easy computation, we get $(s,t)=(1,0)$, or $(2,-1)$. Hence, $\gamma_C=\gamma(\mathcal{O}_C(2H-D))=3$.

$\;$

\noindent{\textbf{Remark 5.2}}. In Example 5.1, $C$ is an aCM curve of genus $9$ and degree $9$. Since the line bundle $\mathcal{O}_C(2H-D)$ computes $\gamma_C$, the second gonality $d_2$ of $C$ is 7. However, in this case, $\rho(9,1,7)=3>0$. 

$\;$


\smallskip

\smallskip


\end{document}